\documentclass[11pt]{amsart}
\usepackage[english]{babel}
\usepackage[usenames,dvipsnames]{color}  
\usepackage{array,mathpazo,amsmath,xcolor}

\usepackage{stmaryrd, mathtools, enumerate, amssymb}

\usepackage[colorlinks=true,linkcolor=NavyBlue, citecolor=LimeGreen, pagebackref]{hyperref} 

\usepackage{ytableau} 

\usepackage{tikz}
\usetikzlibrary{arrows,decorations.pathmorphing,decorations.pathreplacing,positioning,shapes.geometric,shapes.misc,decorations.markings,decorations.fractals,calc,patterns,matrix}

\tikzset{>=stealth',
     cvertex/.style={circle,draw=black,inner sep=1pt,outer sep=3pt},
     vertex/.style={circle,fill=black,inner sep=1pt,outer sep=3pt},
     star/.style={circle,fill=yellow,inner sep=0.75pt,outer sep=0.75pt},
     tvertex/.style={inner sep=1pt,font=\criptsize},
     gap/.style={inner sep=0.5pt,fill=white}}

\newcommand{\arrowrl}[3][20]
{
\hspace{-5pt}
\begin{tikzpicture}
\node (A) at (0,0) {};
\node (B) at (1,0) {};
\draw[->] ($(A)+(0,0.2)$) -- node [above] {$\scriptstyle f^*$} ($(B)+(0,0.2)$);
\draw [->] ($(B)+(0,0.2)$) -- node [below] {$\scriptstyle f_*$} ($(A)+(0,0.2)$);
\end{tikzpicture}
\hspace{-5pt}
}
\newcommand{\adj}[2][20]{\arrowrl}

\tikzset{
    ncbar angle/.initial=90,
    ncbar/.style={
        to path=(\tikztostart)
        -- ($(\tikztostart)!#1!\pgfkeysvalueof{/tikz/ncbar angle}:(\tikztotarget)$)
        -- ($(\tikztotarget)!($(\tikztostart)!#1!\pgfkeysvalueof{/tikz/ncbar angle}:(\tikztotarget)$)!\pgfkeysvalueof{/tikz/ncbar angle}:(\tikztostart)$)
        -- (\tikztotarget)
    },
    ncbar/.default=0.5cm,
}

\tikzset{square left brace/.style={ncbar=0.5cm}}
\tikzset{square right brace/.style={ncbar=-0.5cm}}

\tikzset{round left paren/.style={ncbar=0.5cm,out=120,in=-120}}
\tikzset{round right paren/.style={ncbar=0.5cm,out=60,in=-60}}

\newcommand{\qurep}[2]{\operatorname{\rightleftarrows}^{#1}_{#2}}  

\newcommand{\fm}{{\mathfrak m}}

\newcommand{\fp}{{\mathfrak p}}

\newcommand{\cala}{{\mathcal A}}

\newcommand{\CC}{{\mathbb C}}

\newcommand{\EE}{{\mathbb E}}

\newcommand{\PP}{{\mathbb P}}

\newcommand{\RR}{{\mathbb R}}

\DeclareMathOperator{\End}{ End}

\DeclareMathOperator{\GL}{GL}
\DeclareMathOperator{\gl}{gldim}

\DeclareMathOperator{\Hom}{Hom}

\DeclareMathOperator{\rank}{rank}

\DeclareMathOperator{\Sing}{Sing}
\DeclareMathOperator{\SL}{SL}

\DeclareMathOperator{\Spec}{Spec}

\DeclareMathOperator{\Sym}{{\mathbb S}ym}

\newcommand{{\sbullet}}{{\scriptstyle\bullet}}

\newcommand{\diag}{\mathsf {diag}}

\DeclareMathOperator{\CM}{\ensuremath{ \mathbf{CM}}}
\DeclareMathOperator{\Coh}{\mathbf{Coh}}

\DeclareMathOperator{\fmod}{\ensuremath{ \mathbf{mod}}}

\theoremstyle{definition}
\newtheorem{defn}{Definition}[section]

\newtheorem{rem}[defn]{Remark}
\newtheorem{remark}[defn]{Remark}
\newtheorem{rems}[defn]{Remarks}

\newtheorem{question}[defn]{Question}
\newtheorem{sit}[defn]{}
\newtheorem{example}[defn]{Example}

\newtheorem{exam}[defn]{Example}

\theoremstyle{plain}

\newtheorem{prop}[defn]{Proposition}

\newtheorem{theorem}[defn]{Theorem}
\newtheorem{lem}[defn]{Lemma}

\newtheorem{cor}[defn]{Corollary}

\newtheorem{conjecture}[defn]{Conjecture}

\def\bexa{\begin{exam}}
\def\eexa{\end{exam}}

\def\bpro{\begin{prop}}
\def\epro{\end{prop}}

\def\bcor{\begin{cor}}
\def\ecor{\end{cor}}

\def\bthm{\begin{theorem}}
\def\ethm{\end{theorem}}

\def\bdfn{\begin{eefn}}
\def\edfn{\end{eefn}}

\def\brem{\begin{rem}}
\def\erem{\end{rem}}

\def\brems{\begin{rems}}
\def\erems{\end{rems}}

\def\bsit{\begin{sit}}
\def\esit{\end{sit}}

\def\blem{\begin{lem}}
\def\elem{\end{lem}}

\def\bdi{\pdfsyncstop\begin{eiagram}}
\def\edi{\end{eiagram}\pdfsyncstart}

\def\ba{\begin{array}}
\def\ea{\end{array}}

\def\bnum{\begin{enumerate}}
\def\enum{\end{enumerate}}

\def\be{\begin{equation}}
\def\ee{\end{equation}}

\def\bproof{\begin{proof}}
\def\eproof{\end{proof}}

\newcommand{\rot}[1]{{\color{Red}  #1}}

\begin{document}
\title[Noncommutative resolutions]
{Noncommutative resolutions of discriminants}

\author[Ragnar-Olaf Buchweitz]{Ragnar-Olaf Buchweitz$^\dagger$}
\address{Dept.\ of Computer and Math\-ematical Sciences,
University  of Tor\-onto at Scarborough, 
1265 Military Trail, 
Toronto, ON M1C 1A4,
Canada}

\author{Eleonore Faber}
\address{Department of Mathematics, University of Michigan, Ann Arbor MI, 48109, USA }
\email{emfaber@umich.edu}

\author{Colin Ingalls}
\address{
Department of Mathematics and Statistics,
University of New Brunswick, 
Fredericton, NB. E3B 5A3, 
Canada}
\email{cingalls@unb.ca}

\thanks{
R.-O. \!B.~was partially supported by an NSERC Discovery grant, E.F.~was partially supported by an Oberwolfach Leibniz Fellowship, C.I.~was partially supported by an NSERC Discovery grant.
} 
\thanks{${}^\dagger$The first author passed away on November 11, 2017.}
\date{\today}

\begin{abstract} 
 We give an introduction to the McKay correspondence and its connection to quotients of $\CC^n$ by finite reflection groups. This yields a natural construction of noncommutative resolutions of the discriminants of these reflection groups. This paper is an extended version of E.~F.'s talk with the same title delivered at the ICRA.
\end{abstract}
\maketitle

\section{Introduction}

This article has two objectives: first, we want to present the components of the classical McKay correspondence, which relate algebra, geometry and representation theory. On the other hand, we give a short introduction to discriminants of finite reflection groups and our construction of their noncommutative desingularizations. The details of our new results will be published elsewhere \cite{BFI}.  \\

 This project grew out of the following (which is NOT how we will present our story!):  Start with a commutative ring $R$, then a noncommutative resolution of singularities (=NCR) of $R$ (or of $\Spec(R)$) is an  endomorphism ring $\End_R M$ of a faithful module $M$ such that the global dimension of $\End_R M$ is finite. Endomorphism rings of finite global dimension have got a lot of attention lately because of their connections to various parts of mathematics, such as commutative algebra, noncommutative algebraic geometry and representation theory. The problem of constructing explicit NCRs is difficult in general and one only knows some scattered examples. In particular, NCRs for non-normal rings have not been considered much in the literature, mostly only  in examples where $R$ has Krull-dimension $\leq 2$ (geometrically: $R$ is the coordinate ring of a collection of points, curves or  surfaces).  \\

In \cite{DFI} we were interested in noncommutative resolutions (=NCRs) of certain non-normal hypersurfaces, so-called free divisors, and explicitly constructed a NCR of a normal crossings divisor. After some extensive calculations we realized that there was a more conceptual way to understand this particular NCR as a quotient of the skew group ring of the polynomial ring in $n$ variables and the reflection group $(\mu_2)^n$ by a two-sided ideal generated by an idempotent of the skew group ring. And since the normal crossings divisor is the discriminant of a real reflection group, we had the idea to carry out the same quotient procedure for other discriminants of reflection groups. It was readily verified that indeed this construction works for any discriminant of a finite reflection group. \\
After some more painful calculations and serious doubts, that these quotients are endomorphism rings in general, we were finally able to also prove this for reflection groups generated by reflections of order $2$, and discovered the surprising fact that the corresponding reflection arrangement is in some sense the noncommutative resolution of the discriminant of a reflection group. Moreover, our work generalizes Auslander's theorem about skew group rings of small subgroups $G \subseteq \GL(n,\CC)$ and thus our NCRs give rise to a McKay correspondence for reflection groups. For reflection groups $G \subseteq \GL(2,\CC)$ our results nicely fit in the picture of the classical McKay correspondence: we see that the reflection arrangement provides a representation generator for the torsion-free modules over the discriminant curves. \\

This paper is organized as follows: 
\tableofcontents 

\section{Reflection groups}

Throughout the paper, we assume that $k=\CC$ and that $G \subseteq \GL(V)$ is a finite subgroup acting on $V \cong \CC^n$. Thus,  $G \subseteq \GL(V)$ means $G \subseteq \GL(n,\CC)$. Then $\Sym_k(V) \cong k[x_1, \ldots, x_n]$ is a (graded) polynomial ring, with the standard grading $|x_i|=1$. Assume that $G$ acts on $V$, then any $g \in G$ acts also on $S$ via $g(f(x_1, \ldots, x_n))=f(g(x_1, \ldots, x_n))$. Note that in most of the literature (about the McKay correspondence), one considers the completion $k[[x_1, \ldots, x_n]]$ of $S$, but in fact, all the results we need work for $S$ the graded polynomial ring. \\

Here we gather some well-known results about complex reflection groups, see e.g.~\cite{BourbakiLIE4-6,LehrerTaylor,OTe}, we follow the notation of \cite{OTe}.

\subsection{Basics} An element $g \in \GL(V)$ is called a \emph{pseudo-reflection} if it is conjugate to a diagonal matrix $\diag(\zeta, 1 , \ldots, 1)$ with $\zeta \neq 1$ a root of unity. If $\zeta=-1$, we call $g$ a \emph{(true) reflection}. A pseudo-reflection $g$ fixes a unique hyperplane $H$ pointwise. This hyperplane will be called the \emph{mirror} of $g$.  If $G$ is generated by pseudo-reflections, it is called a \emph{complex reflection group} (or \emph{pseudo-reflection group}). If $G$ is generated by true reflections, then we will call $G$ a \emph{true reflection group}. A finite subgroup $G \subseteq \GL(V)$ is called \emph{small}, if $G$ does not contain any pseudo-reflections.  \\

Let now $G \subseteq \GL(V)$ be finite. The \emph{invariant ring} of the action of $G$ on $V$ is $R=S^G=\{ f \in S: g(f) = f$ for all $g \in G \}$. By the theorem of Hochster--Roberts \cite{HochsterRoberts}, $R$ is Cohen--Macaulay and it is Gorenstein if $G \subseteq \SL(V)$, by \cite{WatanabeGorenstein}. If $G$ is a pseudo-reflection group, we are even in a better situation, cf.~\cite{Che}, \cite[Thm.~6.19]{OTe}:

\begin{theorem}[Chevalley--Shephard--Todd] \label{Thm:CST} 
\begin{enumerate}[(a)]
\item Let $G \subseteq \GL(V)$ be a finite group acting on $V$. Then $G$ is a pseudo-reflection group if and only if there exist homogeneous polynomials $f_1, \ldots, f_n \in R$ such that $R=k[f_1, \ldots, f_n]$. 
\item Moreover, $S$ is free as an $R$-module, more precisely, $S \cong R\otimes_k kG$ as $R$-modules, where $kG$ denotes the group ring of $G$.
\end{enumerate}
\end{theorem}

The polynomials $f_i$ in the theorem are called \emph{basic invariants} of $G$. They are not unique, but their degrees $\deg f_i=d_i$ are uniquely determined by the group $G$. 

\subsection{Discriminant and hyperplane arrangement} Let $G \subseteq \GL(V)$ be any finite subgroup. Then we denote by $V/G$ the quotient by this action. It is an affine variety given by $V/G=\Spec(S^G)$, see e.g.~\cite[Chapter 7]{CLO}. In general $V/G$ is singular, in particular, if $G$ is a small subgroup of $\GL(V)$. If $G$ is a pseudo-reflection group, then by Thm.~\ref{Thm:CST}, $V/G$ is smooth and isomorphic to $V \cong k^n$. We have a natural projection:
$$\pi: V\cong \Spec(S) \longrightarrow V/G \cong \Spec(S^G) \ .$$
On the ring level, this corresponds to the natural inclusion of rings
$$ S^G \cong k[f_1, \ldots, f_n] \hookrightarrow S=k[x_1, \ldots, x_n] \ . $$

 Now let us define two of the central characters in our story: Let $G$ be a finite complex reflection group. The set of mirrors of $G$ is called the \emph{reflection arrangement of $G$}, denoted by $\cala(G)$. Note that $\cala(G) \subseteq V$. Let $H \subseteq V$ be one of these mirrors, that is, $H$ is fixed by a cyclic subgroup generated by a pseudo-reflection $g_H \in G$ of order $\rho_H > 1$, and let $\{l_H =0\}$ be the defining linear equation of $H$. Then the reflection arrangement is given by the (reduced) equation $z:=\prod_{H \subseteq \cala(G)}l_H=0$. The image of the reflection arrangement $\cala(G)$ under the projection $\pi: V \rightarrow V/G$ is called the \emph{discriminant of $G$} and given by a reduced polynomial $\Delta$. Note that $\Delta \in S^G$, whereas $z \not \in S^G$.\\
One can express the discriminant in terms of the basic invariants of $G$: first, one can show that 
$$ J= Jac(f_1, \ldots, f_n)=\det \left(\frac{\partial f_i}{\partial x_j}\right)_{i,j=1, \ldots, n} = u \prod_{H \subseteq \cala(G)}l_H^{\rho_H -1} \ ,$$
where $u \in k^*$. The Jacobian $J$ is an anti-invariant for $G$, that is, $g J = (\det g)^{-1} J$ for any $g \in G$. The discriminant $\Delta$ is then given by
$$ \Delta = Jz = \prod_{H \subseteq \cala(G)}l_H^{\rho_H} \ , $$
which is an element of $S^G$. \\
In the case when $G$ is a true reflection group, we have $\rho_H=2$ for all $H$ and thus $z=J$, and the computation of $\Delta$ simplifies to $\Delta=z^2$. \\

Let us conclude this section with some (non-)examples: 

\begin{example} Let $G \subseteq \SL(V)$, $\dim V=2$, be the cyclic group $C_3$ generated by $\begin{pmatrix} \zeta_3  & 0 \\ 0 & \zeta_3^{-1} \end{pmatrix}$, where $\zeta_3$ is a third root of unity acting on $S=k[x,y]$. Then $V/G$ is a Kleinian surface singularity, denoted by $A_2$. The invariants of $G$ are $u=x^3+y^3$, $v=xy$, $w=x^3-y^3$ and satisfy the relation $w^2-u^2+4v^3$. Thus (after a change of coordinates) $S^G \cong k[u,v,w]/(w^2+u^2+v^3)$. A real picture of $V/G$ can be seen in Fig.~\ref{Fig:A3-nc}. 
\end{example}

\begin{example} \label{Ex:ADEcurves}
Let $G \subseteq \GL(V)$, $\dim V=2$, be a true reflection group. Then the discriminant $\Delta$ is an ADE curve. This can be seen for example by considering $\Gamma:=G\cap \SL_2(V)$. The subgroup $\Gamma$ is of index $2$ in $G$ and $G/\Gamma \cong \mu_2=\langle \sigma \rangle$, where $\sigma$ is a reflection in $G$. The invariant ring $S^\Gamma$ gives rise to an ADE-surface singularity $X:=V/\Gamma$ (cf.~Thm.~\ref{Thm:Klein}). The quotient $X/\langle \sigma \rangle \cong V/G$ is smooth and the branch locus of the natural projection $X \xrightarrow{} X/\langle \sigma \rangle$, is isomorphic to an ADE curve singularity, that is the discriminant $\Delta$ of $G$.  See \cite[Section 3]{KnoerrerCurves} for a more thorough explanation.  \\
In \cite{Bannai} the discriminants of rank two complex reflection groups are calculated explicitly and one sees that indeed all these discriminants are curves of type ADE.
\end{example}

\begin{example}
Let $G=(\mu_2)^n$ be acting on $k^n$ by reflecting about the coordinate hyperplanes. Thus the generators of $G$  are the diagonal matrices
$$S_i = \begin{pmatrix} 1 & 0 & 0 & 0 & 0 \\
0 & \ddots & 0 &0 &0 \\
0 & 0 & -1 &0 &0 \\
0 & 0 & 0 & \ddots & 0 \\
0 & 0 & 0 & 0 & 1
\end{pmatrix} \ .
$$
The hyperplane arrangement $\cala(G)$ is the union of the coordinate hyperplanes. This can easily be seen computationally: The basic invariants in this example are $f_i=x_i^2$, that is, $S^G=k[x_1^2, \ldots, x_n^2]$. Thus $Jac(f_1, \ldots, f_n)=2^n x_1\cdots x_n$ and we can take $z=J=x_1\cdots x_n$. Since $G$ is a true reflection group, $\Delta = z^2 = x_1^2 \cdots x_n^2 =f_1 \cdots f_n$. The coordinate ring of the discriminant is $S^G/(\Delta)=k[f_1, \ldots, f_n]/(f_1 \cdots f_n)$. In this example, the discriminant is also the union of coordinate hyperplanes, but now seen in the quotient $V / G$. See Fig.~\ref{Fig:A3-nc} for a picture in in the $3$-dimensional case.
 \end{example}
 
\begin{figure}[!h]  
\begin{tabular}{c@{\hspace{1.cm}}c}
\includegraphics[width=0.4 \textwidth]{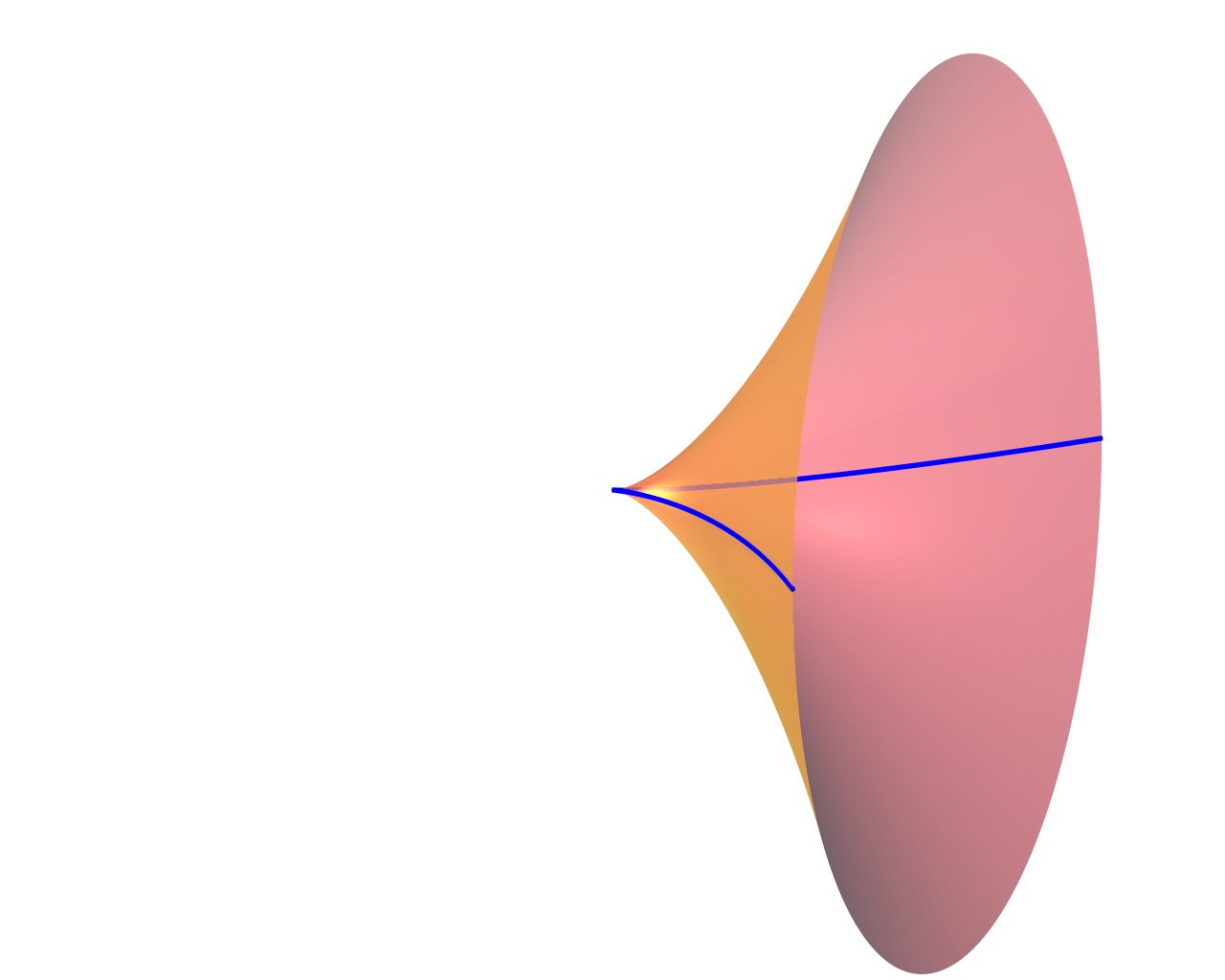}& 
\includegraphics[width=0.4 \textwidth]{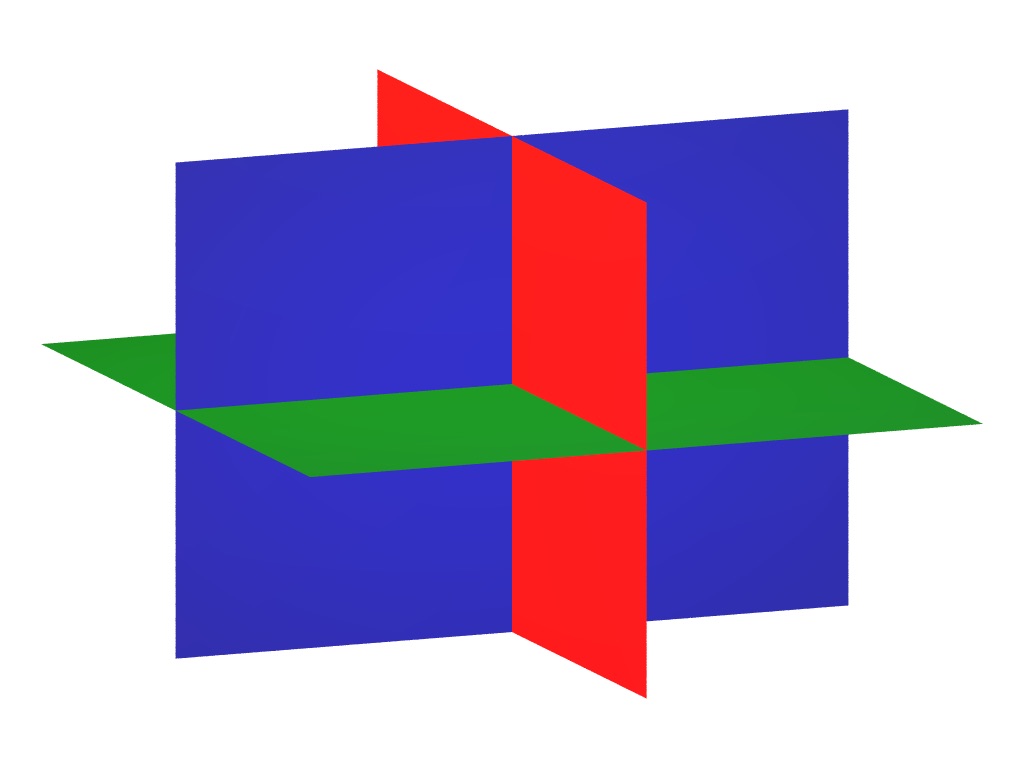}\end{tabular}
\caption{The $A_2$-singularity (left) and the discriminant of $\mu_2^3$.  (right).}
 \label{Fig:A3-nc}
\end{figure}
 
\begin{example}
Let $G=G(2,1,3)$, that is,  the Coxeter group $B_3:=S_3 \ltimes (\mu_2)^3$ (which is also isomorphic to the direct product $S_4 \times \mu_2$) acting on $k^3$ by permutation and sign changes. The invariants of $S=k[x_1,x_2,x_3]$ under this group action are $u=x_1^2+x_2^2+x_3^2$, $v=x_1^4+x_2^4+x_3^4$ and $w=x_1^6+x_2^6+x_3^6$. Since $G(2,1,3)$ is a Coxeter group, it is generated by true reflections, and thus $z=J$. A calculation of the Jacobian determinant of $(u,v,w)$ shows that 
$$ z=J= x_1x_2x_3(x_1^2-x_2^2)(x_1^2-x_3^2)(x_2^2-x_3^2).$$
The invariant ring is $S^G=k[u,v,w]$ and an investigation of $z^2$ shows that it is equal to the product of $w$ with the discriminant of the polynomial $P(t)=t^3+ut^2+vt+w$, see \cite{YanoSekiguchi}. The equation for the discriminant is
$$ \Delta= w(u^2v^2-4v^3-4u^3w+18uvw-27w^2).$$
See Fig.~\ref{Fig:B3} below for pictures of the hyperplane arrangement, given by $z$ and the discriminant $\Delta$ in $\RR^3$. 
\end{example}

\begin{figure}[!h]  
\begin{tabular}{c@{\hspace{1.cm}}c}
\includegraphics[width=0.4 \textwidth]{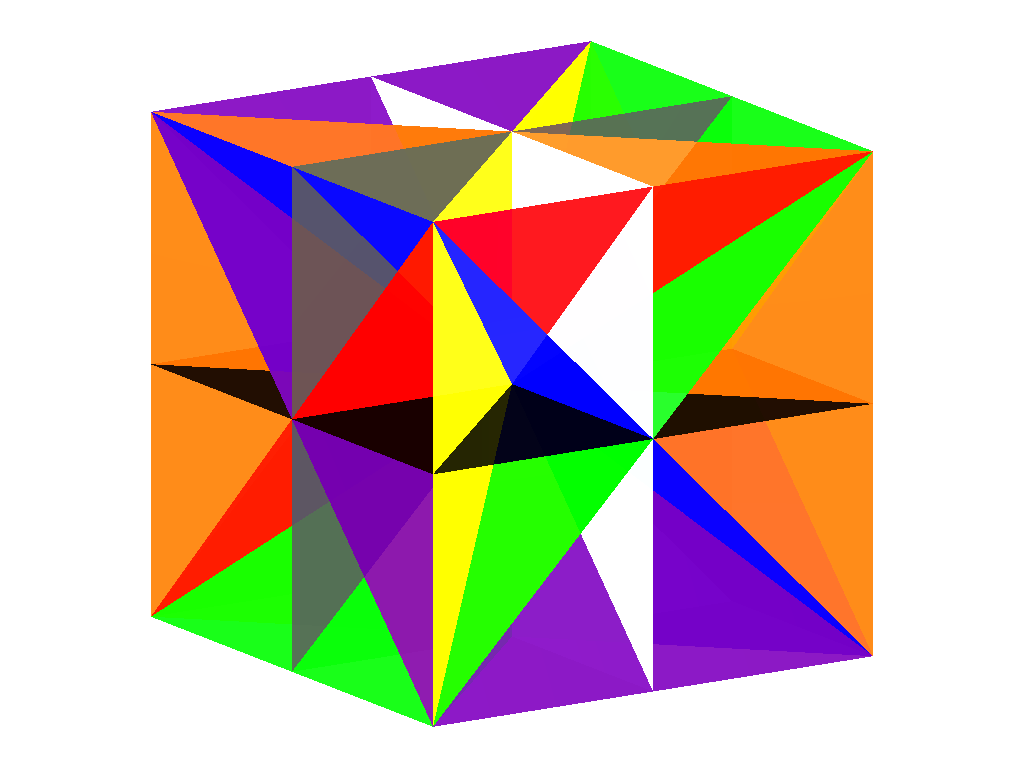}& 
\includegraphics[width=0.4 \textwidth]{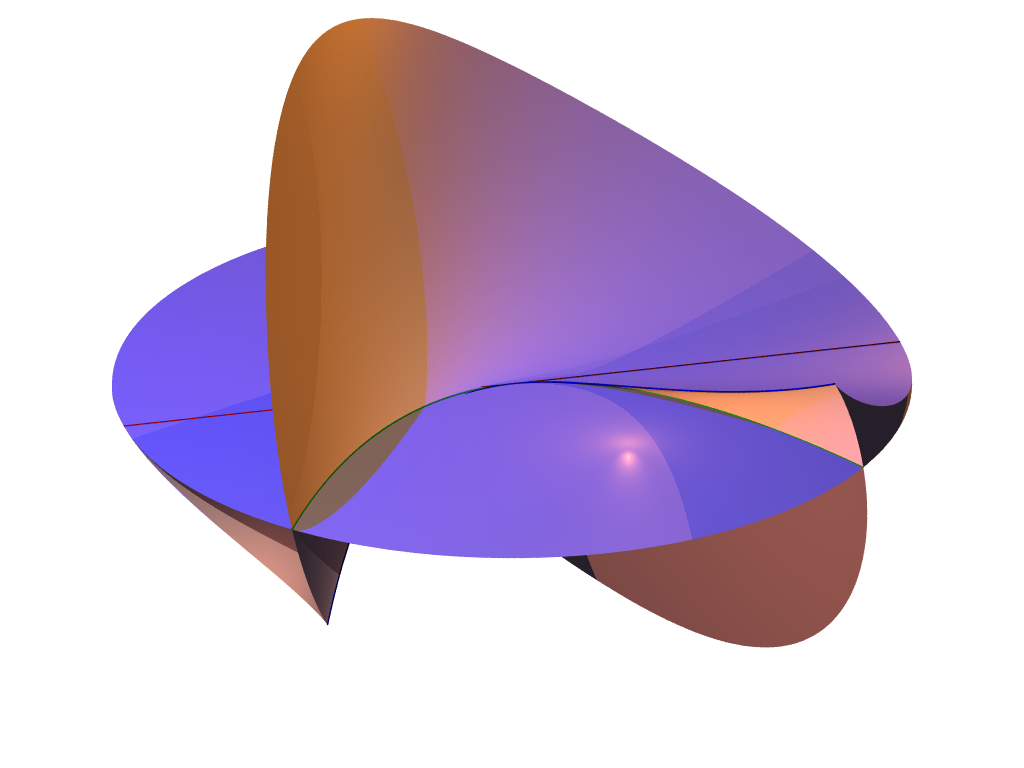}\end{tabular}
\caption{The $B_3$-reflection arrangement (left) and the corresponding discriminant with equation $w(u^2v^2-4v^3-4u^3w+18uvw-27w^2)$  (right).}
 \label{Fig:B3}
\end{figure}

\section{(Noncommutative) resolutions of singularities}

Here we come to the geometric part. Let $X$ be an affine algebraic variety, that is $X=\Spec(R)$, where $R$ is a commutative finitely generated $k$-algebra. If $R$ is not a regular ring, then we say that $X$ is a \emph{singular} algebraic variety and denote by $\Sing X$ its \emph{singular locus}. A \emph{resolution of singularities of $X$} is a proper birational map $\pi: \widetilde X \rightarrow X$ from a smooth scheme $\widetilde X$ to $X$ such that $\pi$ is an isomorphism on the smooth points of $X$.   \\
The central result regarding resolution of singularities is

\begin{theorem}[Hironaka \cite{Hi}]
Let $X$ be a scheme over a field of characteristic $0$ (e.g., $X$ an affine algebraic variety defined over $\CC$). Then there exists a resolution of singularities $\pi: \widetilde X \rightarrow X$.
\end{theorem}

One can impose some extra conditions on $\pi$,  for example that $\pi^{-1}(\Sing(X))$ is a normal crossings divisor in $\widetilde X$ and $\pi$ is a composition of blowups in smooth centers, or (in the case $X$ an irreducible reduced variety over $k$), that $X$ has a so-called embedded resolution, see \cite{Ha4} for more on this. So existence of resolutions of singularities is (with our assumption that $k = \CC$, that is, $k$ of characteristic $0$) not an issue. However, from an algebraic point of view, it would be desirable to have some control over  $\widetilde X$. If $X$ is of Krull-dimension $2$, then there exists a unique \emph{minimal resolution of singularities}, that is, a resolution $Y$ of $X$ such that any other resolution $\widetilde X$ of $X$ factors through $Y$. This is false in higher dimensions, in general there does not even exist a minimal resolution. \\
For some Cohen--Macaulay varieties $X$ (that is: if $X$ is locally $\Spec(R)$, then $R$ is a Cohen--Macaulay ring) there still exist preferred resolutions of singularities: namely those, that do not affect the canonical class of $X$. A resolution of singularities $\pi: \widetilde X \rightarrow X$ is called \emph{crepant} if $\pi^*\omega_X = \omega_{\widetilde X}$. Crepant  resolutions first appeared in the minimal model program, where the term ``crepant'' was coined by M.~Reid to reflect the fact that the pull-back of the canonical sheaf of  $X$ does not pick up discrepancy, see e.g.~\cite{ReidBourbaki}. Crepant resolutions still do not always exist and are usually not isomorphic to each other, but at least their cohomologies should be the same, see \cite{BondalOrlov}:

\begin{conjecture}[Bondal--Orlov] Let $Y$ and $Y^+$ be two crepant resolutions of an algebraic variety $X$. Then there is a derived equivalence of categories
$$D^b(\Coh Y) \simeq D^b(\Coh Y^+) \ ,$$
where $\Coh(-) $ denotes the category of coherent sheaves.
\end{conjecture}

Building on work by Bridgeland--King--Reid for $3$-dimensional quotient singularities \cite{BKR}, Bridgeland proved  \cite{Bridgeland02} the conjecture for complex $3$-dimensional varieties with only terminal singularities. Inspired by their work, Van den Bergh defined a noncommutative crepant resolution (=NCCR) \cite{VandenBergh04}, also see \cite{vandenBerghflops04}: if $X=\Spec(R)$, then the main point in their approach is to construct a (non-commutative) $R$-algebra such that for any crepant resolution $\pi: Y \rightarrow X$ one has $D^b(\Coh Y) \simeq D^b(\fmod A)$. Here $\fmod A$ denotes the category of finitely generated $A$-modules. The algebra $A$ is thus a noncommutative analogue of a crepant resolution of singularities, a NCCR. Later, NCCRs were studied and constructed in various instances, see e.g., \cite{BLvdB10,IyamaWemyss10}. The more general concept of noncommutative resolution (=NCR) was defined by Dao--Iyama--Takahashi--Vial \cite{DaoIyamaTakahashiVial} in 2015. Unfortunately there is (so far) not a good theory about general existence and properties of NCRs, only mostly examples, see e.g. \cite{BurbanIyamaKellerReiten, DFI, DaoIyamaWemyssetc, DoFI, Leuschke07}. A good introduction to NCCRs and categorical geometry with many references can be found in \cite{Leuschke12}. \\ 
But let us state the definitions:

\begin{defn}
Let $R$ be a commutative reduced noetherian ring. An $R$-algebra $A$ is called a \emph{noncommutative resolution (NCR)} of $R$ (or of $\Spec(R)$) if $A$ is of the form $\End_R M$ for $M$ a finitely generated, faithful $R$-module and $\gl (A) < \infty$. The algebra $A$ is called a \emph{noncommutative crepant resolution (NCCR)} of $R$ (or of $\Spec(R)$) if $A$ is a nonsingular order, that is, $\gl A_\fp = \dim R_\fp$ for all $\fp \in \Spec(R)$ and $A$ is a maximal Cohen--Macaulay module over $R$.
\end{defn}

\begin{remark}
In Van den Bergh's original definition, $R$ was assumed to be a commutative normal Gorenstein domain. The reasoning behind this more general definition can be found in \cite{DFI}.
\end{remark}

In general it is not clear how to construct NCRs and whether they exist. However, there is one beautiful example, which will lead to our construction of NCRs for discriminants of reflection groups.

\section{The classical McKay correspondence}

Here we give brief sketches of the components of the classical McKay correspondence. Up to this date, the correspondence has been generalized in various directions, for an account and more detailed references see \cite{BuchweitzMFO}. 
\subsection{Dual resolution graphs} Let $\Gamma \subseteq \SL(V)$ be a finite subgroup acting on $V \cong k^2$.  Denote as above by $S=k[x,y]$ the symmetric algebra of $V$  and the invariant ring $S^\Gamma=R$. Then $X:=V/\Gamma=\Spec(R)$ is a quotient singularity. More precisely, these quotient singularities (aka \emph{Kleinian singularities}) have been classified by Felix Klein \cite{Klein}:

\begin{theorem}[Klein 1884] \label{Thm:Klein}
Let $\Gamma$ and $X$ be as above. Then $X$ is isomorphic to an isolated surface singularity of type ADE, that is, $R \cong k[x,y,z]/(f)$: 
\begin{itemize}
\item $A_n$: \ $f=z^2 + y^2 +x^{n+1}$,  if $\Gamma$ is a cyclic group,
\item $D_n$: \ $f=z^2 + x(y^2+x^{n-2})$ \ for $n \geq 4$, if $\Gamma$ is a binary dihedral group,
\item $E_6$: \ $f=z^2 +x^3 +y^4$, if $\Gamma$ is the binary tetrahedral group,
\item $E_7$: \ $f=z^2 +x(x^2+y^3)$, if $\Gamma$ is the binary octahedral group, 
\item $E_8$: \ $f=z^2 +x^3 + y^5$, if $\Gamma$ is the binary icosahedral group.
\end{itemize}
\end{theorem}

\begin{figure}[!h]   \label{Fig:ADEpics}
\begin{tabular}{c@{\hspace{1.cm}}c}
\includegraphics[width=0.4 \textwidth]{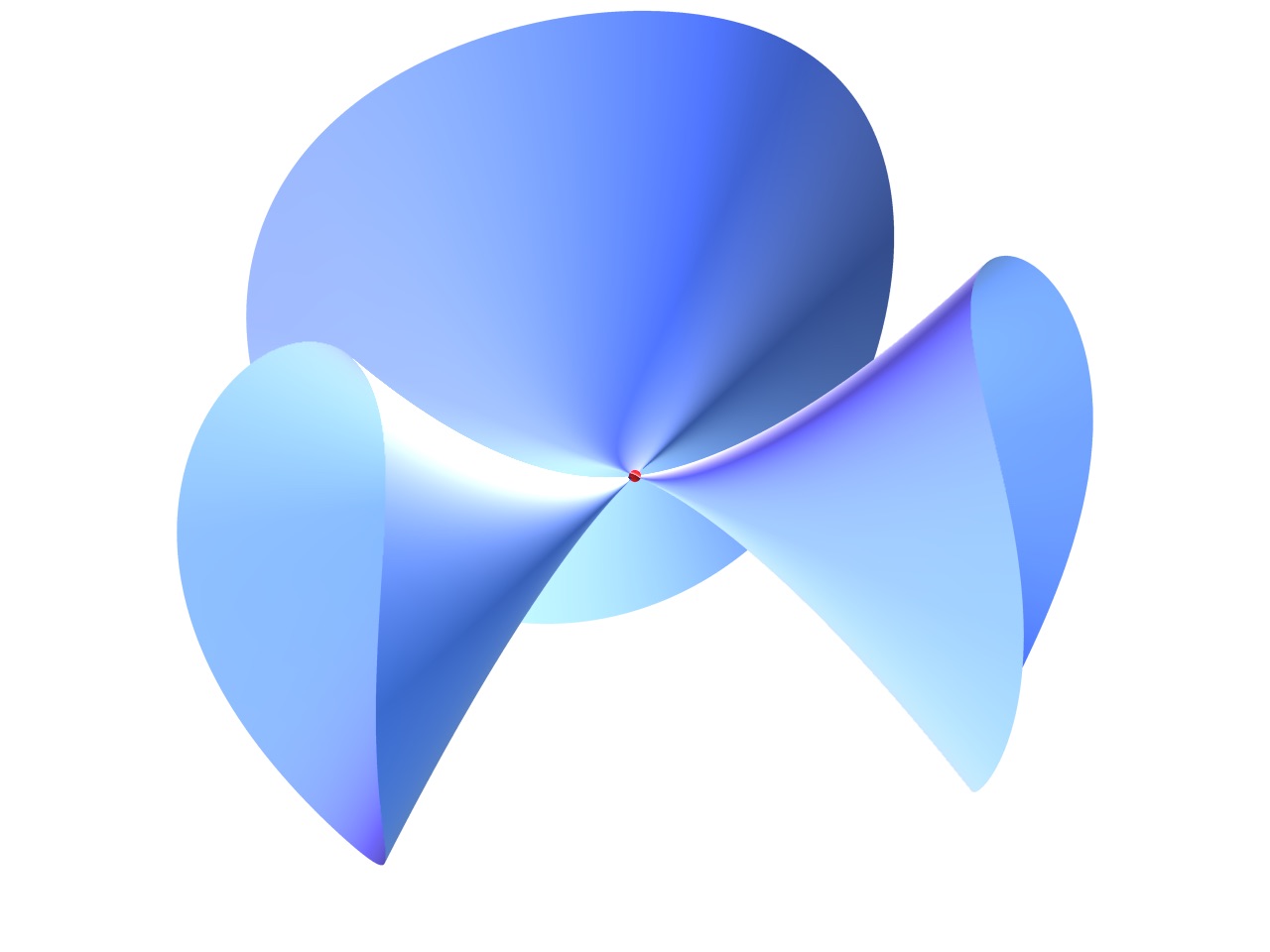}& 
\includegraphics[width=0.4 \textwidth]{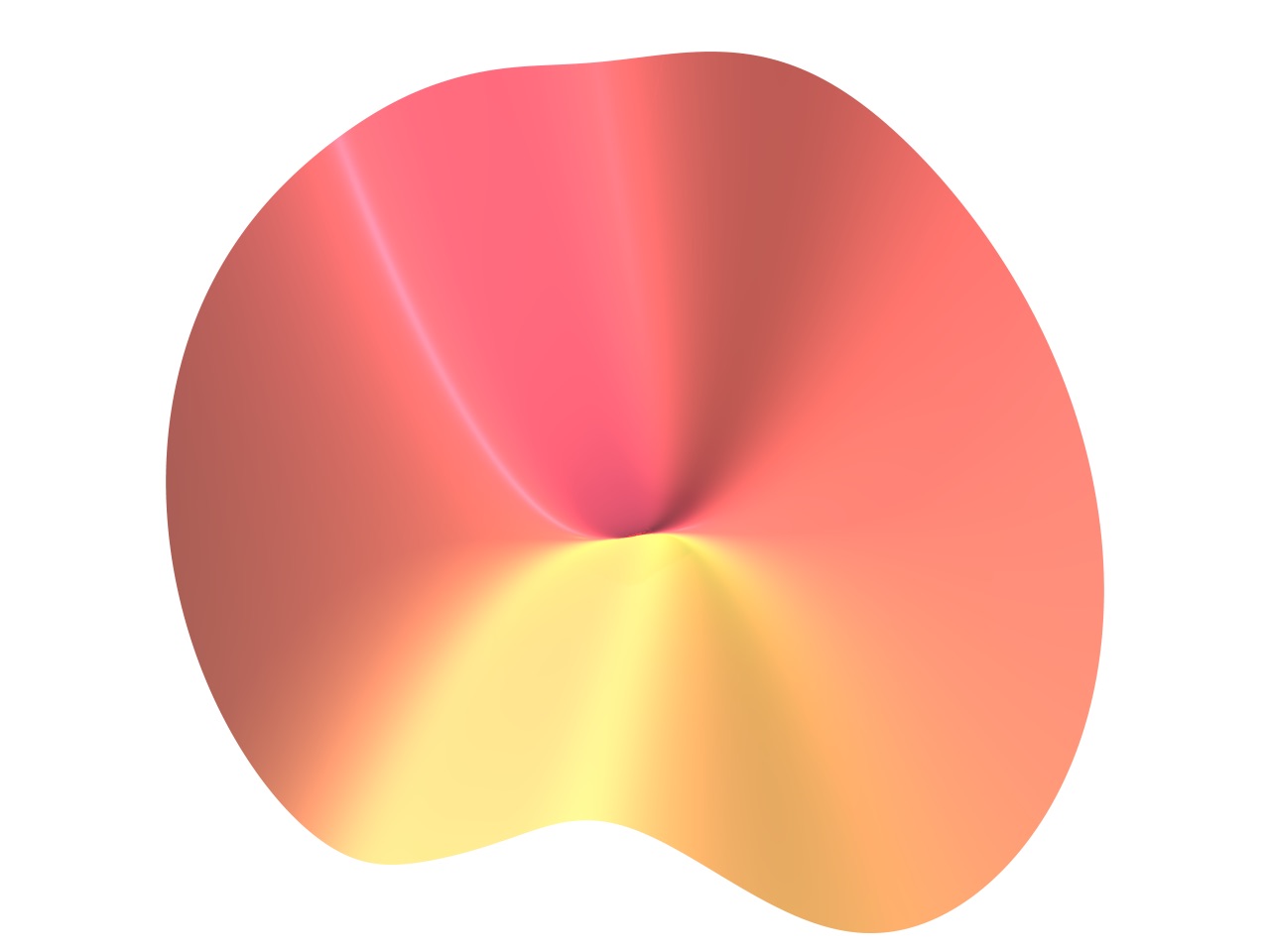}\end{tabular}
\caption{The $D_4$ -singularity with equation $z^2+x(y^2-x^2)=0$ (left) and the $E_8$-singularity with equation $z^2+x^3+y^5=0$ (right).}
\end{figure}

The Kleinian singularities $X$ are also classified by their so-called dual resolution graphs, which happen to be the corresponding ADE-Coxeter-Dynkin-diagrams (for a more detailed account of resolution of surface singularities and especially the rational double points, see \cite{Gre92}, \cite{Durfee}): consider the minimal resolution $\pi: \widetilde X \rightarrow X$ of a Kleinian singularity $X$. The minimal resolution exists, since $X$ is a normal surface singularity and $\widetilde X$ is moreover unique, see \cite{La}. One can for example obtain $\widetilde X$ by successively blowing up singular points. This process can be quite intricate, see for example the dessert in this menu \cite{FaberHauser10}. The preimage of $\{0\} = \Sing(X)$ in $\widetilde X$ will be denoted by $\mathbb{E}=\bigcup_{i=1}^n\EE_i$ and is a union of irreducible curves $\EE_i$ on $\widetilde X$ (this holds for example because one can construct $\widetilde X$ as an embedded resolution of $X \subseteq k^3$, which implies that the preimage of $\Sing X$ must be a normal crossing divisor). Since $X$ is a so-called  rational singularity, the irreducible components $\mathbb{E}_i$ of $\mathbb{E}$ are isomorphic to $\PP^1$'s and intersect each other at most pairwise transversally (this is to say: $\EE$ is a normal crossings divisor on $\widetilde X$). Now one can form the \emph{dual resolution graph} of $X$: the vertices are indexed by the $\mathbb{E}_i$, for $i=1,\ldots,n$ and there is an edge between $\mathbb{E}_i$ and $\mathbb{E}_j$ if and only if $\mathbb{E}_i \cap \mathbb{E}_j \neq \varnothing$. \\
Moreover, one defines the intersection matrix $E:=( (\EE_i \cdot \EE_j))_{i,j=1, \ldots, n}$. Computation shows that $(\mathbb{E}_i \cdot \mathbb{E}_i)=-2$ and $(\EE_i \cdot \EE_j) \in \{0,1\}$ for $i \neq j$ and that $E$ is symmetric and negative-definite. Moreover, there exist divisors  $Z'$ on $\widetilde X$ of the form $Z'=\sum_{i=1}^n m_i \EE_i$ supported on $\EE$ with $m_i > 0$, such that $Z' \cdot \EE_i \leq 0$ for all $i$. The smallest such divisor is denoted by $Z$ and called the \emph{fundamental divisor of $X$}. One decorates the vertices of the dual resolution graph with the multiplicities $m_i$ of $Z$. For an illustration of the $D_4$-singularity see Fig.~\ref{Fig:ADEpics}, and the resolution with the resolution graph are below in Fig.~\ref{Fig:D4resolution}. \\
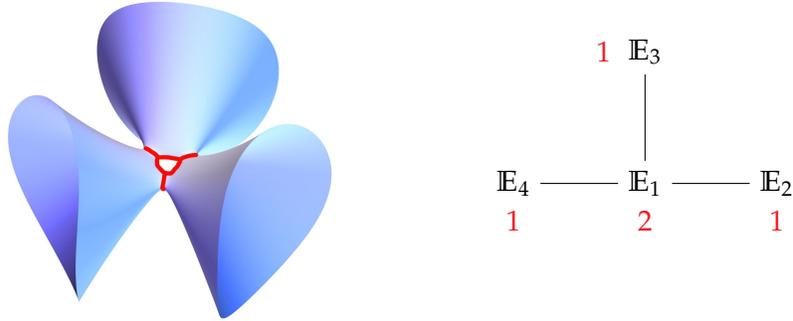
\begin{figure}[!h]   
\begin{tabular}{c@{\hspace{1.cm}}c}
\includegraphics[width=0.4 \textwidth]{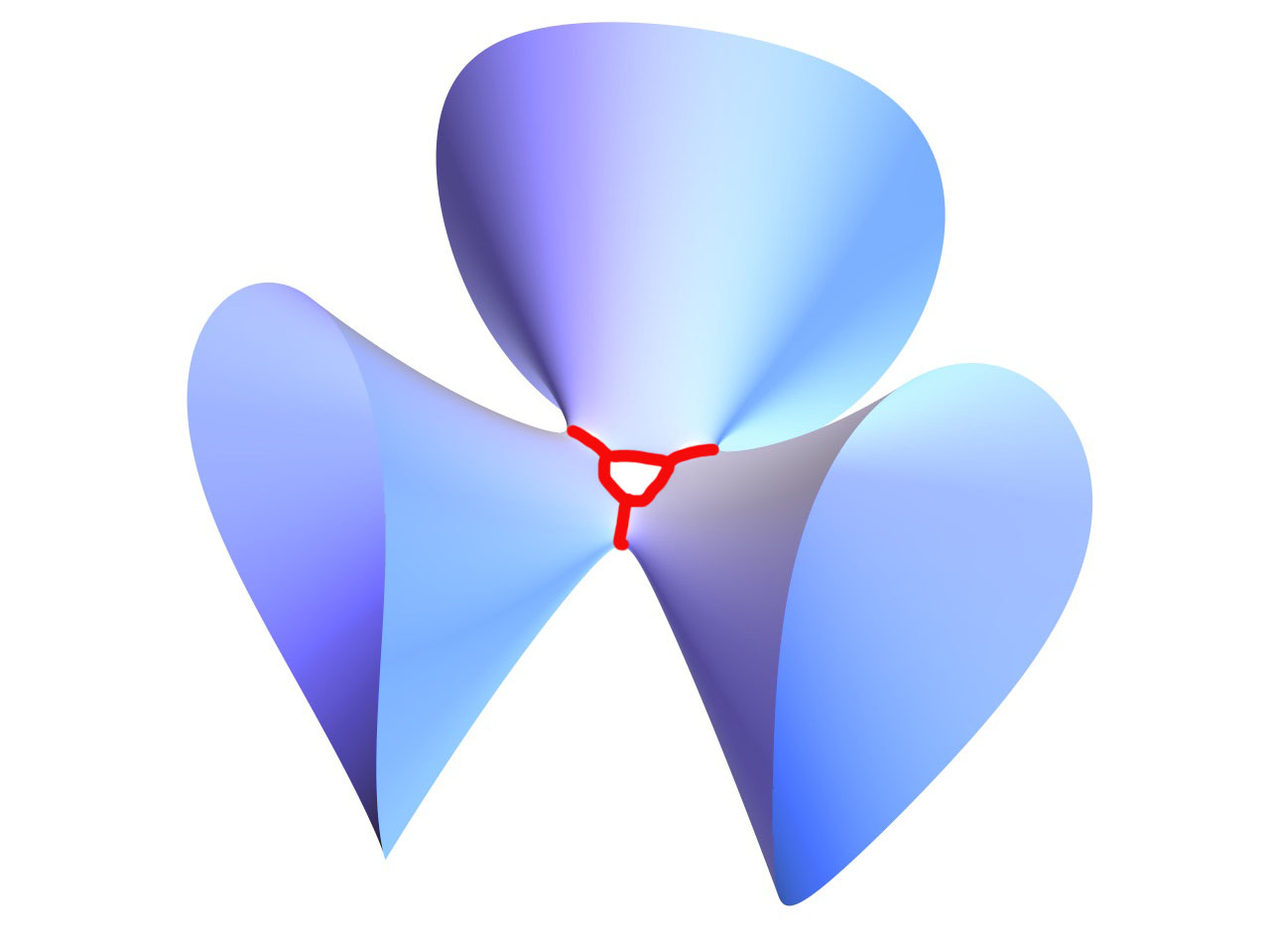}& 

$$
\begin{tikzpicture}
\node at (4,0) {\begin{tikzpicture} 
\node (C1) at (0,0)  {$\EE_4$};
\node (C2) at (1.75,0)  {$\EE_1$};
\node (C3) at (3.5,0)  {$\EE_2$};
\node (C4) at (1.75,1.75)  {$\EE_3$};
\node (C5) at (1.75,-1.75)  { };

\node (C4p) at (1.2,1.75)  {\rot{$1$}};
\node (C2p) at (1.75,-0.5)  {\rot{$2$}};
\node (C3p) at (3.5,-0.5)  {\rot{$1$}};
\node (C1p) at (0,-0.5)  {\rot{$1$}};

\draw [-,bend left=0,looseness=1,pos=0.5] (C1) to node[]  {} (C2);

\draw [-,bend left=0,looseness=1,pos=0.5] (C2) to node[]  {} (C3);

\draw [-,bend left=0,looseness=1,pos=0.5] (C4) to node[]  {} (C2);

\end{tikzpicture}}; 
\end{tikzpicture}
$$
\end{tabular}
\caption{The minimal resolution of the $D_4$-singularity (with the exceptional curves sketched in \rot{red}) and the dual resolution graph of type $D_4$ with the multiplicities of the fundamental divisor in \rot{red}. 
} \label{Fig:D4resolution}
\end{figure}

\subsection{McKay quiver} In 1979, John McKay built the following graph out of a finite subgroup $\Gamma \subseteq \SL(V)$, $\dim V=2$: the embedding of $\Gamma$ in $\SL(V)$ defines the so-called canonical representation $c: \Gamma \hookrightarrow \GL(V)$ of $\Gamma$. Any finite group only has finitely many isomorphism classes of finite dimensional irreducible $k$-representations, given by group homomorphisms $\rho_i: \Gamma \rightarrow \GL(V_i)$, $i=0, \ldots, n$ for vector spaces $V_i$. The trivial representation will be denoted by $\rho_0: \Gamma \rightarrow k^*$, sending any $\gamma$ to $1$. The \emph{McKay quiver of $\Gamma$} consists of the vertices indexed by the $V_i$ and there are $m_{ij}$ arrows from $V_i$ to  $V_j$ if and only if $V_i$ is contained with multiplicity $m_{ij}$ in $V_j \otimes V$. \\

\begin{example}
Consider the group $\Gamma$ generated by 
$$\pm \begin{pmatrix} 1 & 0 \\ 0 & 1 \end{pmatrix}, \pm \begin{pmatrix} i & 0 \\ 0 & -i \end{pmatrix}, \pm \begin{pmatrix} 0 & 1 \\ -1 & 0 \end{pmatrix}, \pm \begin{pmatrix} 0 & i \\ i & 0 \end{pmatrix}.$$
This is the binary dihedral group $D_4$. It has five irreducible representations $\rho_i$, four of which are one-dimensional and one two-dimensional $\rho_1$, which is the canonical representation $c$. Using character theory, one obtains the McKay quiver of $\Gamma$ (the dimensions of the irreducible representations in \rot{red}):
\[
\begin{tikzpicture}
\node at (4,0) {\begin{tikzpicture} 
\node (C1) at (0,0)  {$\rho_0$};
\node (C2) at (1.75,0)  {$\rho_1$};
\node (C3) at (3.5,0)  {$\rho_3$};
\node (C4) at (1.75,1.75)  {$\rho_4$};
\node (C5) at (1.75,-1.75)  {$\rho_2$};

\node (C1p) at (-0.5,0)  {\rot{$1$}};
\node (C2p) at (1.25,0.75)  {\rot{$2$}};
\node (C3p) at (4,0)  {\rot{$1$}};
\node (C4p) at (2.25,1.75)  {\rot{$1$}};
\node (C5p) at (2.25,-1.75)  {\rot{$1$}};

\draw [->,bend left=30,looseness=1,pos=0.5] (C1) to node[]  {} (C2);
\draw [->,bend left=30,looseness=1,pos=0.5] (C2) to node[] {} (C1);

\draw [->,bend left=30,looseness=1,pos=0.5] (C2) to node[]  {} (C3);
\draw [->,bend left=30,looseness=1,pos=0.5] (C3) to node[] {} (C2);

\draw [->,bend left=30,looseness=1,pos=0.5] (C2) to node[]  {} (C5);
\draw [->,bend left=30,looseness=1,pos=0.5] (C5) to node[] {} (C2);

\draw [->,bend left=30,looseness=1,pos=0.5] (C4) to node[]  {} (C2);
\draw [->,bend left=30,looseness=1,pos=0.5] (C2) to node[] {} (C4);

\end{tikzpicture}}; 
\end{tikzpicture}
\] 
\end{example}

The following is due to McKay \cite{McKay79}: \\
\emph{Observation:} Let $\Gamma$ be as above. Then the McKay quiver of $\Gamma$ is connected and contains no loops and $m_{ij}=m_{ji} \in \{0,1\}$. By calculation, one sees that the McKay quiver of $\Gamma$ is the extended Coxeter--Dynkin diagram for $\Gamma$, with arrows in both directions. If one deletes the vertex corresponding to the trivial representation and collapses all arrows $V_i \qurep{}{} V_j$ to an edge, then one obtains the Coxeter--Dynkin diagram associated to $\Gamma$. If one decorates the vertices $V_i$ with the dimensions $\dim V_i$, then one gets back the dual resolution graph of $V/\Gamma$. \\

Thus one obtains a $1-1$-correspondence between the irreducible components of the exceptional divisor of the minimal resolution of the quotient singularity $V/\Gamma$ and the non-trivial irreducible representations of $\Gamma$! This was made precise later by \cite{GonzalezSprinbergVerdier}.  \\

\subsection{AR-quiver} But there is also an algebraic part of the correspondence: Let again $R=S^\Gamma$ for $\Gamma \subseteq \SL(V)$, where $V \cong k^2$. We consider reflexive modules over $R$. Since $R$ is normal, $M$ is reflexive if and only if it is maximal Cohen--Macaulay (=CM), and we write $\CM(R)$ for the category of CM-modules over $R$. By Herzog's theorem \cite{Herzog78}, there is an isomorphism of $R$-modules:
$$ S \cong \bigoplus_{M \in \CM(R)} M^{a_M}, \textrm{ where the integer $a_M=\rank_R M$.} $$ 
Differently phrased: $S$ is a \emph{representation generator} for $\CM(R)$. In particular, there are only finitely many indecomposable objects in $\CM(R)$. One then says that $R$ is of \emph{finite CM-type}. By \cite{BuchweitzGreuelSchreyer, Herzog78, KnoerrerCohenMacaulay}, a Gorenstein ring $R$ is of finite CM-type if and only if it is isomorphic to a simple hypersurface singularity (if $\dim R=2$, these are precisely the ADE-surface singularities).  \\

Starting from the indecomposable CM-modules over $R$, one obtains a third quiver, the \emph{Auslander--Reiten-quiver} (short: \emph{AR-quiver}) of $R$: its vertices are given by the indecomposable CM-modules $M_i$, $i=1, \ldots, n$ of $R$ and there are $m_{ij}$ arrows from $M_i$ to $M_j$ if and only if in the almost split sequence $0 \rightarrow \tau M_j \rightarrow E \rightarrow M_j \rightarrow 0$ ending in $M_j$, $M_i$ appears with multiplicity $m_{ij}$ in $E$. Here one can also see that $\tau M_i = M_i$ for all $M_i \not \cong R$ and that the AR-quiver of $R$ is precisely the McKay quiver of $\Gamma$. Moreover, the ranks of the indecomposable modules $M_i$ correspond to the dimensions of the irreducible representations of $\Gamma$. \\
This astounding correspondence was proven by Maurice Auslander \cite{Auslander86} and looking at his proof more carefully, we will obtain the link to NCCRs. Therefore we need some general notation first: \\

Let $G \subseteq \GL(V)$ finite act on $S=\Sym_k(V)=k[x_1,\ldots, x_n]$ for some vector space $V \cong k^n$. Then define $A=S*G$ to be the \emph{skew group ring} (or \emph{twisted group ring}) of $G$. As an $S$-module, $A$ is just $S \otimes_k kG$, but the multiplication on $A$ is twisted by the action of $G$: for elements $s,s' \in S$ and $g, g' \in G$ define
$$sg \cdot s' g' := (sg(s')) (gg') \ ,$$
and extend by linearity. The following result is the key theorem in proving the correspondence between indecomposable CM-modules and irreducible representations. However note that the assumptions on $G$ are more general than before: we may take any $G$ which does not contain any pseudo-reflections and the dimension of $V$ may be greater than $2$: 

\begin{theorem}[Auslander]
Let $S$ be as above and assume that $G \subseteq \GL(V)$, $\dim V=n$, is small and set $R=S^G$. Then we have an isomorphism of algebras:
$$ A=S * G \xrightarrow{\cong} \End_R(S) \ , sg \mapsto (x \mapsto sg(x)) \ .$$
Moreover, $S* G$ is a CM-module over $R$ and $\gl (S*G)=n$, and $Z(A)=R$.
\end{theorem}

This result can be phrased in terms of NCCRs: 

\begin{cor} If $G\subseteq \SL(V)$, that is, $R$ is a Gorenstein singularity,  then $A$ is an NCCR of $R$. If $G \subseteq \GL(V)$ is small and not in $\SL(V)$, then $A$ is an NCR of $R$. 
\end{cor}

\begin{proof}
The twisted group ring $A \cong \End_RS$ is an endomorphism ring of a faithful CM-module over $R$. Since $\gl A =n < \infty$, $A$ is an NCR. For the crepancy in the $\SL(V)$ case first note that $R$ is a normal Gorenstein domain by the theorems of Hochster--Roberts and Watanabe. Then the fact that $S*G\cong \End_RS$ is CM over $R$ implies that $\End_R S$ is a nonsingular order, see \cite[Lemma 4.2]{VandenBergh04} and thus an NCCR.
\end{proof}

For $G=\Gamma \subseteq \SL(V)$ and $\dim V=2$, Auslander's theorem implies a 1-1-correspondence between the indecomposable CM-modules over $R$, the indecomposable summands of $S$, the indecomposable projective modules over $S*G$, and the irreducible $\Gamma$-representations. Then McKay's observation gives the correspondence to the geometry, namely, the bijection to the irreducible components of the exceptional divisor on the minimal resolution of $R$. \\
We shortly sketch the correspondences: first, by Herzog's theorem the indecomposable objects in $\CM(R)$ are in bijection to the indecomposable $R$-summands of $S$. For the second part, send the indecomposable summand $M$ of $S$ to $\Hom_R(S,M)$, which is an indecomposable projective $\End_R(S)$-module. By Auslander's theorem, it is an indecomposable projective $S*\Gamma$-module. For the last correspondence, send a indecomposable projective $S*\Gamma$-module to $P/(x,y)P$, which is a simple module over $k\Gamma$ and hence an irreducible $\Gamma$-representation. For details about these constructions, see \cite{Auslander86, Yos, LeuschkeWiegand}.  

\subsection{Auslander's theorem and reflections}

Looking at Auslander's theorem, it is natural to ask:
\begin{question} 
Let $G \subseteq \GL(V)$, $\dim V=n$, be generated by pseudo-reflections. Is there an analogue of Auslander's theorem? In particular, is there also a noncommutative resolution hidden somewhere?
 \end{question}
 
 The first problem in generalizing Auslander's theorem is that if $G$ is a pseudo-reflection group, then the map 
 $$ A=S*G \hookrightarrow \End_R(S) $$
 is no longer surjective. One can see this for example for $G=S_n$. Then $\End_R(S)$ contains the so-called Demazure operators, see e.g. \cite{LSch}, which are not contained in the skew group ring. However, one can show in this case that $A$ is the intersection of $\End_R(S) \cap (\End_R(S))^\tau$ in $Q(A)$, the quotient ring of $A$, where $\tau$ is a particular anti-involution on $\End_R(S)$, for details see \cite{KKu1}. \\
 
 Another ``problem'' is that $R=S^G$ is no longer singular  if $G$ is generated by pseudo-reflections, by Thm.~\ref{Thm:CST}. This, however, can be remedied by looking at the discriminant $\Delta \in R$: the discriminant defines a singular hypersurface, whose singularities occur in codimension $1$. More precisely: it was shown by Saito and Terao that the discriminant of a reflection arrangement $\cala(G)$ is a so-called \emph{free divisor} in $R$, that is, the Jacobian ideal of $\Delta$ is a CM-module over $R/(\Delta)$. One can also show that the reflection arrangement itself defines a free divisor in $S$. For more details, see \cite{OTe}. \\
 
The above facts already suggest our strategy to attack the problem of finding an analogue of Auslander's theorem and a noncommutative resolution for the discriminant: we will start with the twisted group ring $S*G$ and using the small group $G \cap \SL(V)$ we will cook up a NCR of $\Delta$, which will also yield a $1\! - \!1$-correspondence to irreducible representations of $G$.

\section{NCRs of discriminants}

For this section let us change the notation slightly: first, we will assume that $G$ is a true reflection group, i.e., $G\subseteq \GL(V)$, $\dim V=n$, is generated by reflections of order $2$ (some of our results also hold for pseudo-reflection groups but the main correspondence only works in this case so far).  We will denote the invariant ring $S^G$ by $T$. By Theorem \ref{Thm:CST} $T \cong k[f_1, \ldots, f_n]$ for the basic invariants $f_i \in S$.
Since $G$ is a true reflection group $J=z$ and the hyperplane arrangement $\cala(G)$ is defined by the polynomial $J=\det(\frac{\partial f_i}{\partial x_j})$. 
The discriminant $\Delta$ is then defined by the equation $J^2=\Delta$, where $\Delta$ is in $T$.  Its coordinate ring will be denoted by $T/(\Delta)$.  
For our construction of the NCR of $\Delta$, we consider the small group $\Gamma = G \cap SL(V)$, with invariant ring $R=S^\Gamma$. We have an exact sequence of groups $1 \rightarrow \Gamma \rightarrow G \rightarrow H \rightarrow 1$, where the quotient is $H \cong \mu_2=\langle \sigma \rangle$.  Moreover, as above, we denote by $A=S*G$ the twisted group ring of $G$. \\

The first result deals with the problem that $A$ is not isomorphic to  $\End_T(S)$. We can show that $A$ is still isomorphic to an endomorphism ring - yet over another twisted group ring:

\begin{theorem} \label{Thm:Bilodeau}
With notation introduced in the paragraph above, there is an isomorphism of rings
$$ \mathrm{End}_{R*H}(S*\Gamma) \cong S*G.$$
\end{theorem}

This result follows from generalizing the approach of H.~Kn\"orrer for curves \cite{KnoerrerCurves} (also using ideas from J.~Bilodeau \cite{Bilodeau}). \\

For simplifying notation set $B:=R*H$. Now one can interpret $B$ as the path algebra over $T$ of a quiver $Q$ modulo some relations. Then $B$ has idempotents $e_{\pm}=\frac{1}{2}(1 \pm \sigma)$, where $\sigma$ is the generator of $H$, as defined above. One crucial observation here is that $B/Be_{\pm}B$ is isomorphic to $T/(\Delta)$ as rings. Thus, using the functor 
$$\mathrm{mod}(B/Be_{-}B) \xleftarrow{i^*(-)= - \otimes_B B/Be_{-}B} \mathrm{mod}(B)$$ 
from the standard recollement we obtain the following result:
\begin{theorem} \label{Thm:Knoerrer}
The functor $i^*$ induces an equivalence of categories 
$$\CM(B)/ \langle e_- B \rangle \simeq \CM(T/(\Delta)), $$
where $\langle e_- B \rangle$ denotes the ideal generated by the object $e_{-}B$ in the category $\CM(B)$.
\end{theorem}

 Theorem \ref{Thm:Bilodeau} shows that $A$ is in some sense too large for being a NCR of $\Delta$. But theorem \ref{Thm:Knoerrer} suggests that we have to quotient $A$ by an ideal: since $A \in \CM(B)$, we calculate its image in $\CM(B)/\langle e_- B \rangle$ and find that it is isomorphic to the quotient $\bar A= A/AeA$, where $e=\frac{1}{|G|}\sum_{g \in G} g\in A$ is the idempotent corresponding to the trivial representation of $G$. On the other hand, calculate the image of $A$ under the recollement functor $i^*$: one gets that
 $$ i^*(A) = \End_{T/(\Delta)}(i^*(S *H)) \ .$$
Exploiting results from Stanley about the structure of $R$ as $T$-module \cite{StanleyInvariants}, we obtain that $i^*(S*H) \cong S/(J)$, as $T/(\Delta)$-module. As a ring, $S/(J)$ is precisely the coordinate ring of the reflection arrangement $\cala(G)$! In the last step, a result of Auslander--Platzeck--Todorov \cite{AuslanderPlatzeckTodorov} about global dimension of quotients yields the following

\begin{theorem} \label{Thm:BFI}
With notation as above,
$$ \bar A:= A/AeA \cong \mathrm{End}_{T/(\Delta)} (S/(J)),$$
and $\gl \bar A < \infty$. Thus $\bar A$ yields a NCR of the free divisor $T/\Delta$. \\
Moreover, the indecomposable projective modules over $\bar A$ are in bijection with the non-trivial representations of $G$ and also with certain CM-modules over the discriminant, namely the $T/(\Delta)$-direct summands of $S/(J)$. 
\end{theorem}

Thus we not only obtain a NCR of discriminants of reflection arrangements but also a McKay correspondence for reflection groups $G$. 

\begin{example} \label{Ex:curves}
Let $\dim V=2$ and let $G \subseteq \GL(V)$ be a true reflection group. Then the discriminant $\Delta$ is an ADE-curve singularity, see example \ref{Ex:ADEcurves}.   We obtain that $S/(J)$ is a generator of the category of CM-modules over $T/(\Delta)$, that is, $\mathrm{add}(S/(J))=\CM(T/(\Delta))$. Moreover, the multiplicities, with which the indecomposables appear in $S/(J)$ are precisely the ranks of the corresponding irreducible representations. \\
As a particular example, take $G=S_3$, whose discriminant $\Delta$ is a cusp (aka $A_2$-singularity). For the cusp, one can show that $S/(J)\cong T/(\Delta) \oplus \fm^{\oplus 2}$ as $T/(\Delta)$-modules, where $\fm$ denotes the maximal ideal $(u,v)$ in $T/(\Delta)=k[u,v]/(u^3-v^2)$. Here one sees that $T/(\Delta)$ corresponds to the determinantal representation and $\fm$ to the canonical representation given by $G \hookrightarrow \GL(2,k)$. 

\end{example}

\section{Further questions}

In general we are still missing a conceptual explanation for the $T/(\Delta)$-direct summands of $S/(J)$: in dimension $2$, see example \ref{Ex:curves}, any indecomposable CM-module over $T/(\Delta)$ can be obtained from $S/(J)$ (via matrix factorizations).  But if $\dim V \geq 3$, that is, the Krull-dimension of $T/(\Delta)$ is greater than or equal to $2$, $T/(\Delta)$ is not of finite CM-type. However, the factors of $S/(J)$ are still CM-modules and can be calculated as matrix factorizations. So far we only have few examples  (the normal crossing divisor, as studied in Dao--Faber--Ingalls \cite{DFI} or the swallowtail, that is, the discriminant of $S_4$, where we can describe the $T/(\Delta)$-direct summands using Hovinen's classification of graded rank $1$ MCM-modules \cite{Hovinen}). \\

Thus, for a geometric interpretation of the direct summands of $S/(J)$ we want to establish a similar correspondence as in \cite{GonzalezSprinbergVerdier}. A next step would be to realize geometric resolutions as moduli spaces of isomorphism classes of representations of certain  algebras as in \cite{CrawleyBoevey}.\\

But, in a different vein, the structure of $\bar A$ as an $S/(J)$-module seems to be easier to understand: $\bar A$ is isomorphic to the cokernel of the map $\varphi$ given by left multiplication on $G$, i.e., the matrix of $\varphi$ corresponds to the multiplication table of the group $G$. Since  Frobenius and Dedekind, the block decomposition of this matrix is well-known. Our goal is to use this surprising discovery to learn more about the decomposition of $\bar A$ and hence of $S/(J)$ over the discriminant. \\
Moreover, looking at the quiver of $\bar A$, the exact form of the generating (necessarily quadratic) relations remains mysterious and will be the subject of further research.

\section{Acknowledgements}
We want to thank the organizers of the ICRA for making it such a pleasant and stimulating conference. In particular, we thank Graham Leuschke for his encouragement to write this proceedings article. We also thank the anonymous referee for helpful comments.
We acknowledge and highly appreciate that the research presented in this paper was supported through the program ``Research in Pairs'' in 2015 and the program ``Oberwolfach Leibniz Fellowships'' in 2016 by the Mathematisches Forschungsinstitut Oberwolfach and by the Institut Mittag-Leffler (Djursholm, Sweden) in frame of the program ``Representation theory'' in 2015. \\

E.F. and C.I. want to thank R.-O.B. for the many hours we spent together working on this project - his enthusiasm, knowledge, and deep insights will be dearly missed.


\begin{thebibliography}{BLvdB10}

\bibitem[APT92]{AuslanderPlatzeckTodorov}
M.~Auslander, M.~I. Platzeck, and G.~Todorov.
\newblock Homological theory of idempotent ideals.
\newblock {\em Trans. Amer. Math. Soc.}, 332(2):667--692, 1992.

\bibitem[Aus86]{Auslander86}
M.~Auslander.
\newblock {Rational singularities and almost split sequences}.
\newblock {\em Transactions of the AMS}, 293(2):511--531, 1986.

\bibitem[Ban76]{Bannai}
Etsuko Bannai.
\newblock Fundamental groups of the spaces of regular orbits of the finite
  unitary reflection groups of dimension {$2$}.
\newblock {\em J. Math. Soc. Japan}, 28(3):447--454, 1976.

\bibitem[BFI16]{BFI}
R.-O. Buchweitz, E.~Faber, and C.~Ingalls.
\newblock {A McKay correspondence for reflection groups}.
\newblock 2017.
\newblock {\tt arXiv:1709.04218}.

\bibitem[BGS87]{BuchweitzGreuelSchreyer}
R.-O. Buchweitz, G.-M. Greuel, and F.-O. Schreyer.
\newblock Cohen-{M}acaulay modules on hypersurface singularities. {II}.
\newblock {\em Invent. Math.}, 88(1):165--182, 1987.

\bibitem[BIKR08]{BurbanIyamaKellerReiten}
I.~Burban, O.~Iyama, B.~Keller, and I.~Reiten.
\newblock Cluster tilting for one-dimensional hypersurface singularities.
\newblock {\em Adv. Math.}, 217(6):2443--2484, 2008.

\bibitem[Bil05]{Bilodeau}
J.~Bilodeau.
\newblock Auslander algebras and simple plane curve singularities.
\newblock In {\em Representations of algebras and related topics}, volume~45 of
  {\em Fields Inst. Commun.}, pages 99--107. Amer. Math. Soc., Providence, RI,
  2005.

\bibitem[BKR01]{BKR}
T.~Bridgeland, A.~King, and M.~Reid.
\newblock {The MacKay correspondence as an equivalence of derived categories}.
\newblock {\em J. Amer. Math. Soc.}, 14(3):535--554, 2001.

\bibitem[BLvdB10]{BLvdB10}
R.-O. Buchweitz, G.~J. Leuschke, and M.~van~den Bergh.
\newblock {Non-commutative desingularization of determinantal varieties I}.
\newblock {\em Invent. Math.}, 182(1):47--115, 2010.

\bibitem[BO02]{BondalOrlov}
A.~Bondal and D.~Orlov.
\newblock Derived categories of coherent sheaves.
\newblock In {\em Proceedings of the {I}nternational {C}ongress of
  {M}athematicians, {V}ol. {II} ({B}eijing, 2002)}, pages 47--56. Higher Ed.
  Press, Beijing, 2002.

\bibitem[Bou81]{BourbakiLIE4-6}
Nicolas Bourbaki.
\newblock {\em \'El\'ements de math\'ematique}.
\newblock Masson, Paris, 1981.
\newblock Groupes et alg\`ebres de Lie. Chapitres 4, 5 et 6. [Lie groups and
  Lie algebras. Chapters 4, 5 and 6].

\bibitem[Bri02]{Bridgeland02}
T.~Bridgeland.
\newblock {Flops and derived categories}.
\newblock {\em Invent. Math.}, 147(3):613--632, 2002.

\bibitem[Buc12]{BuchweitzMFO}
R.-O. Buchweitz.
\newblock {From platonic solids to preprojective algebras via the McKay
  correspondence}.
\newblock {\em Oberwolfach Jahresbericht}, pages 18--28, 2012.

\bibitem[CB00]{CrawleyBoevey}
W.~Crawley-Boevey.
\newblock On the exceptional fibres of {K}leinian singularities.
\newblock {\em Amer. J. Math.}, 122(5):1027--1037, 2000.

\bibitem[Che55]{Che}
C.~Chevalley.
\newblock Invariants of finite groups generated by reflections.
\newblock {\em Amer. J. Math.}, 77:778--782, 1955.

\bibitem[CLO07]{CLO}
D.~Cox, J.~Little, and D.~O'Shea.
\newblock {\em Ideals, Varieties, and Algorithms}.
\newblock Springer, New York, 3 edition, 2007.

\bibitem[DFI15]{DFI}
H.~Dao, E.~Faber, and C.~Ingalls.
\newblock Noncommutative (crepant) desingularizations and the global spectrum
  of commutative rings.
\newblock {\em Algebr. Represent. Theory}, 18(3):633--664, 2015.

\bibitem[DFI16]{DoFI}
B.~Doherty, E.~Faber, and C.~Ingalls.
\newblock Computing global dimension of endomorphism rings via ladders.
\newblock {\em J. Algebra}, 458:307--350, 2016.

\bibitem[DII{\etalchar{+}}16]{DaoIyamaWemyssetc}
H.~Dao, O.~Iyama, S.~Iyengar, R.~Takahashi, M.~Wemyss, and Y.~Yoshino.
\newblock Noncommutative resolutions using syzygies.
\newblock 2016.
\newblock {\tt arXiv:1609.04842}.

\bibitem[DITV15]{DaoIyamaTakahashiVial}
H.~Dao, O.~Iyama, R.~Takahashi, and C.~Vial.
\newblock Non-commutative resolutions and {G}rothendieck groups.
\newblock {\em J. Noncommut. Geom.}, 9(1):21--34, 2015.

\bibitem[Dur79]{Durfee}
A.-H. Durfee.
\newblock {Fifteen characterizations of rational double points}.
\newblock {\em L'Enseignement Math{\' e}matique}, 25:131--163, 1979.

\bibitem[FH10]{FaberHauser10}
E.~Faber and H.~Hauser.
\newblock Today's menu: geometry and resolution of singular algebraic surfaces.
\newblock {\em Bull. Amer. Math. Soc.}, 47(3):373--417, 2010.

\bibitem[Gre92]{Gre92}
G.-M. Greuel.
\newblock {Deformation und Klassifikation von Singularit{\"a}ten und Moduln}.
\newblock {\em Deutsch. Math.-Verein. Jahresber. Jubil{\"a}umstagung 1990},
  pages 177--238, 1992.

\bibitem[GSV83]{GonzalezSprinbergVerdier}
G.~Gonzalez-Sprinberg and J.-L. Verdier.
\newblock Construction g\'eom\'etrique de la correspondance de {M}c{K}ay.
\newblock {\em Ann. Sci. \'Ecole Norm. Sup. (4)}, 16(3):409--449 (1984), 1983.

\bibitem[Hau03]{Ha4}
H.~Hauser.
\newblock {The Hironaka Theorem on resolution of singularities}.
\newblock {\em Bull. Amer. Math. Soc.}, 40:323--403, 2003.

\bibitem[Her78]{Herzog78}
J{\"u}rgen Herzog.
\newblock Ringe mit nur endlich vielen {I}somorphieklassen von maximalen,
  unzerlegbaren {C}ohen-{M}acaulay-{M}oduln.
\newblock {\em Math. Ann.}, 233(1):21--34, 1978.

\bibitem[Hir64]{Hi}
H.~Hironaka.
\newblock Resolution of singularities of an algebraic variety over a field of
  characteristic zero.
\newblock {\em Ann. Math.}, 79:109--326, 1964.

\bibitem[Hov09]{Hovinen}
B.~Hovinen.
\newblock Matrix factorizations of the classical determinant.
\newblock Phd thesis, University of Toronto, 2009.

\bibitem[HR74]{HochsterRoberts}
Melvin Hochster and Joel~L. Roberts.
\newblock Rings of invariants of reductive groups acting on regular rings are
  {C}ohen-{M}acaulay.
\newblock {\em Advances in Math.}, 13:115--175, 1974.

\bibitem[IW14]{IyamaWemyss10}
O.~Iyama and M.~Wemyss.
\newblock Maximal modifications and {A}uslander-{R}eiten duality for
  non-isolated singularities.
\newblock {\em Invent. Math.}, 197(3):521--586, 2014.

\bibitem[KK86]{KKu1}
Bertram Kostant and Shrawan Kumar.
\newblock The nil {H}ecke ring and cohomology of {$G/P$} for a {K}ac-{M}oody
  group {$G$}.
\newblock {\em Proc. Nat. Acad. Sci. U.S.A.}, 83(6):1543--1545, 1986.

\bibitem[Kle93]{Klein}
F.~Klein.
\newblock {\em {Vorlesungen \"uber das Ikosaeder und die Aufl\"osung der
  Gleichungen vom f\"unften Grade}}.
\newblock Birkhauser Verlag, Berlin, 1993.
\newblock Reprint of the 1884 original, edited, with an introduction and
  commentary by Peter Slodowy.

\bibitem[Kn{\"o}84]{KnoerrerCurves}
H.~Kn{\"o}rrer.
\newblock The {C}ohen--{M}acaulay modules over simple hypersurface
  singularities.
\newblock 1984.
\newblock Max-Planck Institut f{\"u}r Mathematik und Sonderforschungsbereich
  theoretische Physik MPI/SFB 84-51.

\bibitem[Kn{\"o}87]{KnoerrerCohenMacaulay}
Horst Kn{\"o}rrer.
\newblock Cohen-{M}acaulay modules on hypersurface singularities. {I}.
\newblock {\em Invent. Math.}, 88(1):153--164, 1987.

\bibitem[Lau71]{La}
H.~Laufer.
\newblock {\em Normal Two-dimensional Singularities}, volume~71 of {\em Annals
  of Math. Studies}.
\newblock Princeton Univ. Press, 1971.

\bibitem[Leu07]{Leuschke07}
G.~J. Leuschke.
\newblock Endomorphism rings of finite global dimension.
\newblock {\em Canad. J. Math.}, 59(2):332--342, 2007.

\bibitem[Leu12]{Leuschke12}
G.~J. Leuschke.
\newblock Non-commutative crepant resolutions: scenes from categorical
  geometry.
\newblock In {\em Progress in commutative algebra 1}, pages 293--361. de
  Gruyter, Berlin, 2012.

\bibitem[LS92]{LSch}
Alain Lascoux and Marcel-Paul Sch{\"u}tzenberger.
\newblock D\'ecompositions dans l'alg\`ebre des diff\'erences divis\'ees.
\newblock {\em Discrete Math.}, 99(1-3):165--179, 1992.

\bibitem[LT09]{LehrerTaylor}
G.~I. Lehrer and D.~E. Taylor.
\newblock {\em Unitary reflection groups}, volume~20 of {\em Australian
  Mathematical Society Lecture Series}.
\newblock Cambridge University Press, Cambridge, 2009.

\bibitem[LW12]{LeuschkeWiegand}
G.~J. Leuschke and R.~Wiegand.
\newblock {\em Cohen-{M}acaulay representations}, volume 181 of {\em
  Mathematical Surveys and Monographs}.
\newblock American Mathematical Society, Providence, RI, 2012.

\bibitem[McK80]{McKay79}
John McKay.
\newblock Graphs, singularities, and finite groups.
\newblock In {\em The {S}anta {C}ruz {C}onference on {F}inite {G}roups ({U}niv.
  {C}alifornia, {S}anta {C}ruz, {C}alif., 1979)}, volume~37 of {\em Proc.
  Sympos. Pure Math.}, pages 183--186. Amer. Math. Soc., Providence, R.I.,
  1980.

\bibitem[OT92]{OTe}
P.~Orlik and H.~Terao.
\newblock {\em Arrangements of hyperplanes}, volume 300 of {\em Grundlehren der
  Mathematischen Wissenschaften}.
\newblock Springer-Verlag, Berlin, 1992.

\bibitem[Rei02]{ReidBourbaki}
M.~Reid.
\newblock La correspondance de {M}c{K}ay.
\newblock {\em Ast\'erisque}, (276):53--72, 2002.
\newblock S\'eminaire Bourbaki, Vol. 1999/2000.

\bibitem[Sta77]{StanleyInvariants}
R.~P. Stanley.
\newblock Relative invariants of finite groups generated by pseudoreflections.
\newblock {\em J. Algebra}, 49(1):134--148, 1977.

\bibitem[VdB04a]{VandenBergh04}
M.~Van~den Bergh.
\newblock Non-commutative crepant resolutions.
\newblock In {\em The legacy of Niels Henrik Abel}, pages 749--770. Springer,
  Berlin, 2004.

\bibitem[VdB04b]{vandenBerghflops04}
M.~Van~den Bergh.
\newblock Three-dimensional flops and noncommutative rings.
\newblock {\em Duke Math. J.}, 122(3):423--455, 2004.

\bibitem[Wat74]{WatanabeGorenstein}
Keiichi Watanabe.
\newblock Certain invariant subrings are {G}orenstein. {I}, {II}.
\newblock {\em Osaka J. Math.}, 11:1--8; ibid. 11 (1974), 379--388, 1974.

\bibitem[Yos90]{Yos}
Y.~Yoshino.
\newblock {\em Cohen--Macaulay modules over Cohen--Macaulay rings}, volume 146
  of {\em London Mathematical Society Lecture Note Series}.
\newblock Cambridge University Press, Cambridge, 1990.

\bibitem[YS81]{YanoSekiguchi}
Tamaki Yano and Jir\=o Sekiguchi.
\newblock The microlocal structure of weighted homogeneous polynomials
  associated with {C}oxeter systems. {II}.
\newblock {\em Tokyo J. Math.}, 4(1):1--34, 1981.

\end{thebibliography}
\newcommand{\etalchar}[1]{$^{#1}$}
\def\cprime{$'$} \def\cprime{$'$}

\end{document}